\documentclass[reqno]{amsart}

\usepackage{amsthm,amsmath,amssymb,mathrsfs,graphicx,color,url}

\theoremstyle{plain}
\newtheorem{thm}{Theorem}[section]
\newtheorem{prop}[thm]{Proposition}
\newtheorem{lem}[thm]{Lemma}
\newtheorem{cor}[thm]{Corollary}

\theoremstyle{definition}

\theoremstyle{remark}

\newtheorem{rem}[thm]{Remark}

\numberwithin{equation}{section}

\def\hyp{\mathrm{hyp}}
\def\res{\mathrm{Res}}

\def\one{\mathbf{1}}

\def\t{\mathbf{t}}
\def\z{\mathbf{z}}
\def\w{\mathbf{w}}
\def\ztilde{\tilde{\mathbf{z}}}
\def\atilde{\tilde{\mathbf{a}}}
\def\wbar{\bar{w}}
\def\a{\mathbf{a}}
\def\b{\mathbf{b}}
\def\mutilde{\tilde{\mu}}
\def\zbar{\bar{z}}
\def\wbar{\bar{w}}
\def\disp{\displaystyle} 

\def\bra{\langle}
\def\ket{\rangle}

\def\phi{\varphi} 

\newcommand{\newoperator}[2]{\DeclareMathOperator{#1}{#2}}
\newoperator{\supp}{supp}
\newoperator{\tr}{Tr}
\newoperator{\re}{Re}
\newoperator{\im}{Im}
\newoperator{\sgn}{sgn}
\newoperator{\pf}{Pf}
\newoperator{\per}{per}
\newoperator{\var}{var}

\def\Z{\mathbb{Z}}
\def\E{\mathbb{E}}
\def\C{\mathbb{C}}
\def\D{\mathbb{D}}
\def\R{\mathbb{R}}
\def\P{\mathbb{P}}
\def\T{\mathbb{T}}
\def\X{\mathbb{X}}
\def\cZ{\mathcal{Z}}
\def\cU{\mathcal{U}}

\def\cH{\mathcal{H}}

\def\cP{\mathcal{P}}

\def\cS{\mathcal{S}}

\begin{document} 
\title[Spectral representation of correlation functions for
zeros]{Spectral representation of correlation functions for
zeros of Gaussian power series with stationary
coefficients}
\author[T.~Shirai]{Tomoyuki SHIRAI}
\date{
\noindent
{\the\year/\the\month/\the\day}}  
\address{Institute of Mathematics for Industry, Kyushu University, 744 Motooka, Nishi-ku, Fukuoka 819-0001, Japan}
\email{shirai@imi.kyushu-u.ac.jp} 
\begin{abstract} 
We analyze Gaussian analytic functions (GAFs) defined as
 power series with coefficients modeled by discrete
 stationary Gaussian processes, utilizing their spectral
 measures. We revisit some limit theorems for random
 analytic functions and examine some examples of GAFs
 through numerical computations. Furthermore, we provide an
 integral representation of the $n$-point correlation
 functions of the zero sets of GAFs in terms of
 the spectral measures of the underlying coefficient Gaussian
 processes.
\end{abstract}
\maketitle 

\section{Introduction}  \label{sec:intro}

Gaussian analytic functions (GAFs) have attracted significant
interest and have been extensively studied by many
researchers
(cf. \cite{Kahane85,Sodin00,Sodin-Tsirelson-04, HKPV09} and
references therein). 
One of the striking results in the
study of GAFs is that the zeros of random power series $\sum_{n=0}^{\infty}
 \zeta_n z^n$ with i.i.d. complex Gaussian coefficients 
 $\{\zeta_n\}_{n=0}^{\infty}$, whose covariance kernel is
 the Szeg\H{o} kernel, form the determinantal point
 process (DPP) on the unit disk associated with the Bergman
 kernel $(1-z\wbar)^{-2}$ \cite{Peres-Virag05} (see for details on DPPs (cf. \cite{ST00, So00, ST03a, HKPV09}). 
As a natural generalization of this model, various related
models have been studied. 
The singular points of random matrix-valued Gaussian analytic
 functions are shown to form the DPPs associated with weighted Bergman kernel $(1-z\wbar)^{-\alpha}$ 
 \cite{K09}. 
The zeros of random power series $\sum_{n=0}^{\infty} 
 \zeta^{\R}_n z^n$ with i.i.d. \textit{real} Gaussian coefficients 
 $\{\zeta^{\R}_n\}_{n=0}^{\infty}$, which is the limiting
 Kac random polynomial, form the Pfaffian point
 process on the unit disk \cite{Forrester10, MS13}. 
The Gaussian Laurant series whose covariance kernel is 
the weighted Szeg\H{o} kernel on an annulus was studied in \cite{KS22a}. 
We are concerned with the zeros of Gaussian
power series $\sum_{n=0}^{\infty} \xi_n z^n$ with coefficients being
stationary, centered, complex Gaussian 
process $\{\xi_n\}_{n=0}^{\infty}$. 
In \cite{Noda-Shirai23, Shirai25+}, we discussed the expected number of
zeros within a disk and the density of zeros in the radial
direction, observing their asymptotic behavior as the
radius approaches the circle of convergence. 
Our observations reveal that the 
spectral measure of the coefficient Gaussian process plays a
crucial role in determining the asymptotic behavior of the
number of zeros near the radius of convergence. 
From the viewpoint of the spectral measure, the zeros of Gaussian
analytic functions in a strip in the complex plane, with
translation-invariant distribution, are studied in
\cite{Feldheim13,Feldheim-Dvora18}. 
This line of study traces back to the work of Paley and
Wiener \cite{PW87}.
In \cite{ANS17,BSW18, BBS21}, the authors investigate random
power series and their zeros from the perspective of
spectral measures (or limiting empirical spectral measures in the case
of deterministic coefficients), considering cases where the coefficients are not necessarily Gaussian processes but
weakly stationary processes, and even in certain cases where
the coefficients are deterministic. 

In this paper, we also emphasize the spectral perspective to study the zeros of GAFs.
The remainder of this paper is organized as follows. 
In Section~\ref{sec:GAFandspec}, we review the
Edelman-Kostlan formula for the $1$-point correlation
function and give its spectral representation. We also recall
some results of limit theorems for random analytic funcitons 
established in \cite{Shirai12}. 
In Section~\ref{sec:examples}, we explore examples of
GAFs associated with three different spectral measures and present some numerical
computations. In Section~\ref{sec:n-correlation}, we provide
a spectral representation of the $n$-point correlation
function for the zeros of a GAF whose covariance kernel is a
stationary complex Gaussian process. As an example of this 
represetation result, we compute certain quantities related to the
i.i.d. coefficients case to illustrate the Peres-Vir\'ag theorem.  

\section{Gaussian analytic functions and their spectral
 representation} \label{sec:GAFandspec}
\subsection{Spectral representation of GAF and the intensity
  of its zeros} 

Let $\Xi = \{\xi_k\}_{k \in \Z}$ be a stationary, centered, complex Gaussian process with covariance 
\[
 \gamma(k-l) = \E[\xi_k \overline{\xi_l}].  
\]
Throughout the paper, we assume that $\gamma(0)=1$, i.e., the
distribution of each marginal $\xi_k$ is 
complex standard normal. 
We consider the Gaussian power series with dependent
coefficients $\Xi=\{\xi_k\}_{k \in \Z}$ defined by  
\[
 X_{\Xi}(z) := \sum_{k=0}^{\infty} \xi_k z^k
\]
and its zeros. 
Since the radius of convergence of
$X_{\Xi}(z)$ is equal to $1$ a.s. under $\gamma(0)=1$, 
the zeros of $X_{\Xi}(z)$ form a point process
$\cZ_{X_{\Xi}} := \sum_{z \in \D : X_{\Xi}(z) =0} \delta_z$
on $\D:=\{z \in \C : |z| < 1\}$. 
We denote the covariance kernel of the GAF $\X = \{X_{\Xi}(z)\}_{z \in \D}$ by $K(z,w)$. 
Then, we have the spectral representation of $K(z,w)$
\cite{Shirai25+} as 
\begin{equation}
 K(z,w) = \E[X_{\Xi}(z) \overline{X_{\Xi}(w)}] 
= \int_{\T} s(z,t) \overline{s(w,t)} F(dt), 
\label{eq:Szw}
\end{equation}
where $\T=(-\pi, \pi]$, the kernel $s(z,t)$ is given by the
Szeg\H{o} kernel 
\begin{equation}
 s(z,t) = \frac{1}{1-ze^{-it}} \quad (z \in \D, t \in \T)  
\label{eq:kzt} 
\end{equation}
and $F(dt)$ is the spectral measure of the coefficient
Gaussian process $\Xi$ \cite{DM76}, i.e., 
$F(dt)$ is a \textit{probability} measure on $\T$ (since
$\gamma(0)=1$) such that 
\[
 \gamma(k) = \int_{\T} e^{-ikt} F(dt). 
\]
We denote the right-hand side of \eqref{eq:Szw} by $K_F(z,w)$. 
Since $|s(z,t)|^2 = (1-2r\cos (\theta-t) + r^2)^{-1}$ for $z=re^{i\theta}$, 
the covariance kernel $K_F(z,z)$ for $z = re^{i\theta}$ can be
expressed as 
\begin{equation}
 K_F(z,z) = \frac{1}{1-r^2} \int_{\T} P(r,\theta-t) F(dt),
\label{eq:KFzz} 
\end{equation}
where $P(r,\theta)$ is the Poisson kernel given by 
\begin{equation}
 P(r,\theta) = \frac{1-r^2}{1-2r \cos \theta +r^2}. 
\label{eq:poisson-kernel} 
\end{equation}

The first intensity (or $1$-point correlation function) 
of the zeros of a GAF can be computed from the
covariance kernel $K(z,w)$ by the
Edelman-Kostlan formula \cite{EK95} as follows: 
\begin{thm}[Edelman-Kostlan] \label{thm:edelman-kostlan}
Let $X=(X(z))_{z \in D}$ be a GAF with convariance $K(z,w) =
 \E[X(z)\overline{X(w)}]$. Then, the first intensity of
 zeros of $X(z)$ is given by the formula 
\[
 \rho_1(z) = \frac{1}{\pi} \partial_z \partial_{\wbar} \log
 K(z,w)\Big|_{w=z}. 
\]
\end{thm}
When $K(z,w)$ is of the form \eqref{eq:Szw}, using the fact
that $\partial_z s(z,t) = e^{-it}s(z,t)^2$ and
the symmetrization of the variable $t$, we can see that 
\begin{equation}
 \rho_1(z) = \frac{1}{2\pi K(z,z)^2} \int_{\T^2} 
|e^{-it}-e^{-iu}|^2 |s(z,t)s(z,u)|^4
F(dt)F(du) \quad (z \in \D). 
\label{eq:spectral-rep-density}
\end{equation}
In Section~\ref{sec:n-correlation}, we will provide an integral
representation of the $n$-point correlation function of the zeros. 

\subsection{Convergence of random analytic functions and
  their zero processes}
Let $D\subset \C$ be a domain and $\cH(D)$ be the space of
complex analytic functions on $D$ equipped with the metric 
\[
 \rho(f,g) = \sum_{j=1}^{\infty} \frac{1}{2^j}
 \frac{\|f-g\|_{K_j}}{1+\|f-g\|_{K_j}},  
\] 
where $\|f\|_K$ is the supremum norm of $f$ on a compact set
$K$ and $K_1 \subset K_2 \subset \dots \nearrow D$ is an
exhaustion of $D$ by compact sets $K_j$. The metric space
$(\cH(D), \rho)$ turns out to be complete and separable. 
We denote the totality of probability
measures on $\cH(D)$ by $\cP(\cH(D))$. 
We call an $\cH(D)$-valued random variable a \textit{random analytic
function}. We recall the following fact established in 
\cite[Proposition 2.5 with Remark on page 341]{Shirai12}). 
\begin{prop}\label{prop:tightness}
Let $D\subset \C$ be a domain and $\{X_n\}_{n=1}^{\infty}$ be a sequence of random analytic
 functions on $D$. If $\sup_n \E[|X_n(z)|^p]$ is locally
 integrable on $D$ for some $p>0$, then $\{\mu_{X_n}\}_{n=1}^{\infty}$ is tight in
 $\cP(\cH(D))$. Moreover, if $\{X_n\}_{n=1}^{\infty}$
 converges to $X$ in the sense of finite-dimensional
 distributions, then $\{\mu_{X_n}\}$ converges weakly to $\mu_X$ as
 $n \to \infty$, or equivalently, $\{X_n\}_{n=1}^{\infty}$
 converges to $X$ in law. 
\end{prop}
We also recall the convergence of zero processes \cite[Proposition 2.3]{Shirai12}. 
\begin{prop}\label{prop:zeros-converge}
A sequence of random analytic functions
 $\{X_n\}_{n=1}^{\infty}$ 
converges in law to $X$. Then, the zero process $\cZ_{X_n}$
 converges in law to $\cZ_X$ provided that $X \not\equiv
 0$ almost surely. 
\end{prop}

We apply these results to our current situation. 
Let $X_F = \{X_F(z)\}_{z \in \D}$ be the GAF corresponding
to the covariance kernel 
\begin{equation}
 K_F(z,w) = \int_{\T} s(z,t) \overline{s(w,t)} F(dt).  
\label{eq:cov=KF}
\end{equation}

\begin{cor}\label{cor:weak-conv}
Let $F_n(dt), n=1,2,\dots$ and $F(dt)$ be probability
measures on $\T$ and suppose that $\{F_n(dt)\}_{n=1}^{\infty}$
converges weakly to $F(dt)$.  
Then, $X_{F_n}$ corresponding to $F_n$ converges in law to
 $X_F$ corresponding to $F$. Moreover, the zero process
 $\cZ_{X_{F_n}}$ converges in law to that of $\cZ_{X_F}$. 
\end{cor}
\begin{proof}
The local integrability of $\sup_n \E[|X_{F_n}(z)|^2]$
 follows from the fact that $\E[|X_{F_n}(z)|^2] = K_{F_n}(z,z) \le (1-|z|)^{-2}$
 for every $n$. Since $s(z,t)$ is bounded continuous in $t$
 for every $z \in \D$, by the weak convergence of
 $\{F_n\}_{n=1}^{\infty}$ to $F$, $K_{F_n}(z,w)$ converges to
 $K_{F}(z,w)$ for every $z,w \in \D$. This implies that
 $X_{F_n}$ converges to $X_{F}$ in the sense of 
 finite-dimensional distributions. 
Therefore, $X_F$ converges in law to $X$ by Proposition~\ref{prop:tightness}. 
The zero processes $\cZ_{F_n}, n=1,2,\dots$ converges in law
 to $\cZ_F$ Proposition~\ref{prop:zeros-converge} since $X_F
 \not\equiv 0$ almost surely. 
\end{proof}

\section{Examples} \label{sec:examples}
Using Propositions~\ref{prop:tightness}, \ref{prop:zeros-converge} and Corollary~\ref{cor:weak-conv},
we discuss three examples. 
\subsection{The case $F(dt) = dt/(2\pi)$} 
In this case, we easily see that $\gamma_F(k) = \delta_{k,0}$
and the kernel $K_F(z,w)$ is given by 
\[
K_F(z,w) = \int_{\T} s(z,t) \overline{s(w,t)} \frac{dt}{2\pi} = 
 \frac{1}{1-z\wbar} \quad 
 (z, w \in \D), 
\]
which is also called the Szeg\H{o} kernel. 
The kernel $s(z,t)$ is the boundary value of $(1-z\wbar)^{-1}$ at $w=e^{it}$. 
This spectral measure corresponds to the
i.i.d. coefficients case. 
This GAF is sometimes called the hyperbolic GAF and here denoted
by $X_{\hyp}(z)$. 
The zeros of $X_{\hyp}(z)$ form the DPP associated 
with Bergman kernel $\pi^{-1} (1-z\wbar)^{-2}$
\cite{Peres-Virag05}.  
In Subsection~\ref{subsec:hyperbolicGAF}, we will see 
a proof of this fact as an example of the spectral representation of 
$n$-point correlation function of the zeros of GAF (Theorem~\ref{thm:n-correlation}). 
The expected number of zeros in a domain $D \subset \D$ is
the maximum among the GAFs with the covariance of the form
\eqref{eq:cov=KF} \cite{Noda-Shirai23}. 
In Fig.~\ref{fig:fig1}, we show the zeros of the approximate polynomial $X_{\hyp}^{(N)}(z) := \sum_{k=0}^N \zeta_k
z^k$. The covariance function is given by 
\begin{equation}
 K_F^{(N)}(z,w) = \sum_{k=0}^N (z\wbar)^k =
 \frac{1-(z\wbar)^{N+1}}{1-z\wbar}. 
\label{eq:ex1cov} 
\end{equation}
By Theorem~\ref{thm:edelman-kostlan} with \eqref{eq:ex1cov}, the density of zeros inside $\D$ is given by 
\begin{equation}
\rho_{1,F}^{(N)}(z) 
= \frac{1}{\pi} \left\{
\frac{1}{(1-|z|^2)^2} - \frac{(N+1)^2|z|^{2N}}{(1-|z|^{2(N+1)})^2}
\right\}. 
\label{eq:rho1FN} 
\end{equation}
The expected number of zeros within $\D_r=\{z \in \C : |z|<r\}$ is 
\[
 \int_{\D_r} \rho_{1,F}^{(N)}(z) dz
= \int_0^{r^2} \Big(\frac{1}{(1-t)^2} - \frac{(N+1)^2
t^{N}}{(1-t^{N+1})^2} \Big) dt
= \frac{r^2}{1-r^2} - \frac{(N+1)r^{2(N+1)}}{1-r^{2(N+1)}}.  
\]
Then, by taking the limit $r \to 1$, we see that 
\[
 \int_{\D} \rho_{1,F}^{(N)}(z) dz = \frac{N}{2}. 
\] 
On average, half of the zeros of $X_{\hyp}^{(N)}(z)$
are inside $\D$, and half are outside. See the left panel of
Fig.~\ref{fig:fig1}. 
Since $K_F^{(N)}(z,z) \le K_F(z,z)$, 
$\sup_N \E{|X_{\hyp}^{(N)}(z)|^2}$ is locally bounded, and
also $K_F^{(N)}(z,w)$ converges to $K_F(z,w)$. Then, by
Propositions~\ref{prop:tightness} and \ref{prop:zeros-converge}, 
we have that $X_{\hyp}^{(N)}(z)$ converges to $X_{\hyp}(z)$ in law as $N
\to \infty$, and its zero process restricted on $\D$ converges in law to that
of $X_{\hyp}(z)$, the DPP associated with Bergman kernel.

\begin{figure}[htbp]
\begin{center}
\includegraphics[width=0.9 \hsize]{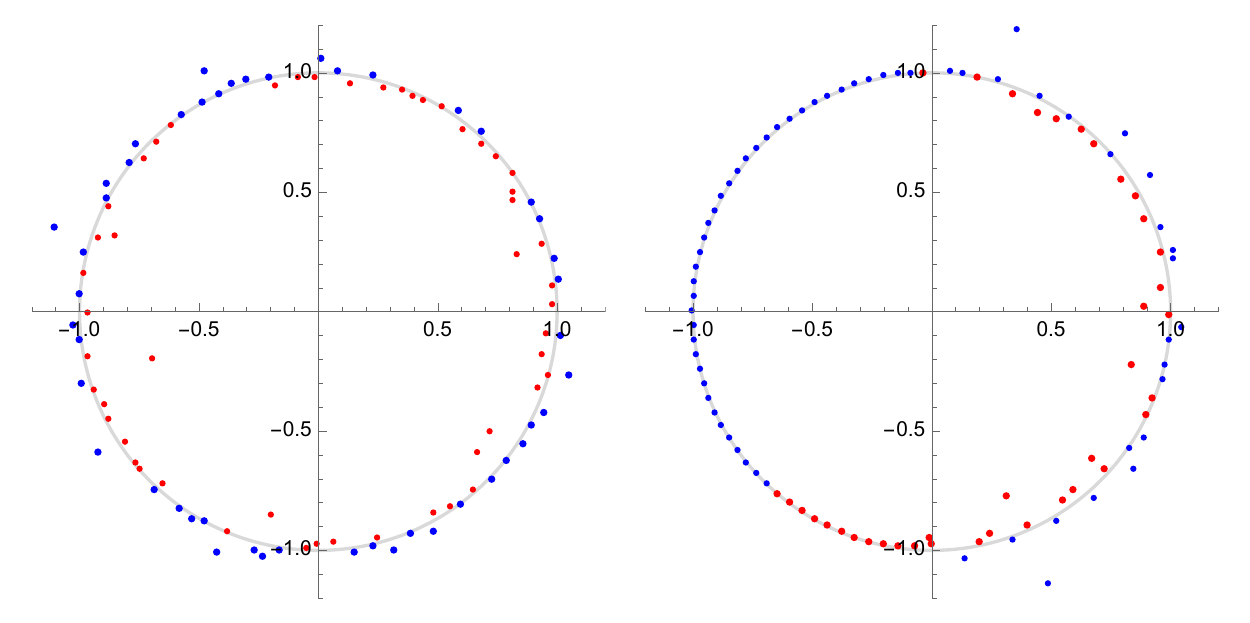}
\end{center}
\caption{Left: The zeros of the approximate polynomial $X_{\hyp}^{(100)}(z) =
 \sum_{k=0}^{100} \xi_k z^k$ for $F(dt) = dt/(2 \pi)$. Red points indicate zeros
 inside the unit disk, while blue points indicate zeros
 outside the unit disk. 
Right: The zeros of $X_F(z)$ for
 $F(dt)=\one_{[-\pi/2,\pi/2]}(t)dt/\pi$. The zeros in the right
 half are distributed randomly, similar to those in the left panel,
 while the zeros in the left half are neatly aligned near
 the unit circle. The latter corresponds to the zeros that might 
 disappear in the limit as $N \to \infty$, which is
 consistent with the fact that the density of the zeros on
 the left is $O(1)$.} 
\label{fig:fig1}
\end{figure}

\subsection{The case $F_n(dt) =\frac{1}{n}
  \sum_{k=0}^{n-1} \delta_{2k \pi/n}(dt)$}

We consider the case where $F_n(dt)$ is atomic and 
supported on the $n$-points $\{e^{2k\pi i/n}\}_{k=0}^{n-1}$. 
In this case, 
\[
 \gamma_{F_n}(k) \equiv 
\begin{cases}
 1 & \text{$k \equiv 0 \mod n$} \\
 0 & \text{$k \not\equiv 0 \mod n$} 
\end{cases} 
\iff  \xi_{pn+k} = \xi_k
\]
where $\xi_0, \xi_1, \dots, \xi_{n-1}$ are
independent, and $\{\xi_k\}_{k=0}^{\infty}$ is periodic with
period $n$. 
The covariance kernel of $X_{F_n}(z)$ is given by 
\begin{equation}
 K_{F_n}(z,w) = \frac{1}{(1-z^n)(1-\wbar^n)} \cdot
 \frac{1-(z\wbar)^n}{1-z\wbar}. 
\label{eq:ex2cov} 
\end{equation}
It is easy to see that $X_{F_n}(z)$ is meromorphic
and its meromorphic extension is expressed as  
\begin{equation}
 X_{F_n}(z) = \frac{1}{1-z^n} \sum_{k=0}^{n-1} \xi_k z^k. 
\label{eq:XFn} 
\end{equation}
Therefore, $X_{F_n}(z)$ has poles at 
$z=e^{2\pi i k/n} \ (k=0,1,\dots,n-1)$ and the zeros 
at those of the random polynomial $\sum_{k=0}^{n-1} \xi_k
z^k$ in $\C$, which is the same as in the previous example
with $N=n-1$. 
This is evident from the form of $X_{F_n}(z)$ in
\eqref{eq:XFn}, but when applying the Edelman-Kostlan
formula to \eqref{eq:ex2cov}, the preceding factor in
\eqref{eq:ex2cov} vanishes due to the differentiation
$\partial_z \partial_{\wbar}$ and does not
contribute, so the density of zeros inside $\D$ remains the same as in
the previous example and is given \eqref{eq:rho1FN} with $N=n-1$. 
Half of zeros of the meromorphic extension appear inside
$\D$ on average. 
We consider the finite approximation $X_{F_n}^{(N)}(z)$ up to the degree $nN-1$ 
\[
X_{F_n}^{(N)}(z) = \sum_{k=0}^{nN-1} \xi_k z^k
= \frac{1-z^{nN}}{1-z^n} \sum_{k=0}^{n-1} \xi_k z^k 
= (1-z^{nN}) X_{F_n}(z). 
\]
Therefore, there are zeros $\exp(2\pi i k /(nN)),
 k=0,1,\dots,nN-1$ on the unit circle as well as the zeros
 of $X_{F_n}(z)$. 
These virtual zeros that appear in the finite
 approximation are not observed in the limit $N \to
 \infty$. 

As $n \to \infty$, 
$F_n(dt) = \frac{1}{n} \sum_{k=0}^{n-1} \delta_{2k \pi
 /n}(dt)$ converges to 
the Lebesgue probability measure $F(dt) =
 dt/(2\pi)$.
By Corollary~\ref{cor:weak-conv}, $X_{F_n}$ on
 $\D$ converges in law to $X_F = X_{\hyp}$ as $n \to \infty$.

\subsection{The case $F(dt) = \one_{[-\pi/2, \pi/2]}(t) dt/\pi$} 
When $F(dt) = \one_{[-\pi/2, \pi/2]}(t) dt/\pi$, the covariance
function is given by 
\[
 \gamma_F(k) = 
\begin{cases}
 1 & k=0, \\
\frac{2}{\pi} \frac{\sin(k \pi /2)}{k} & k \not=0. 
\end{cases}
\]
In this case, the coefficient Gaussian process $\{\xi_k\}_{k \ge 0}$
exhibits long-range correlation. 
Indeed, it is known that a stationary Gaussian process $\{\xi_k\}_{k
\in \Z}$ is purely-nondeterministic if and only if $F(ds)$
is absolutely continuous and its density $f(s)=F(ds)/ds$ satisfies 
$\int_{-\pi}^{\pi} \log f(s)ds > -\infty$, which clearly fails in
this case. This example is also discussed in \cite[Example 2.4]{Shirai25+}. 

From \eqref{eq:KFzz}, the covariance kernel is given as the harmonic extension of
the spectral density $f(s)$, i.e., for $z=re^{i\theta}$,  
\begin{align*}
K_F(z,z) 
 &= \frac{1}{1-r^2}\frac{1}{2\pi} \int_{-\pi/2}^{\pi/2} 
 P(r,\theta-t) dt \\
&= \frac{1}{1-r^2} \left\{\frac{1}{2} + \frac{1}{\pi} \arctan\Big( 
\frac{2r}{1-r^2} \cos \theta\Big)\right\},  
\end{align*}
where $P(r,\theta)$ is the Poisson kernel in
\eqref{eq:poisson-kernel}.  
Here we carefully used the formula 
\[
 \frac{d}{dt} 2 \arctan\left(\frac{1+r}{1-r}
 \tan\frac{t}{2} \right)
= P(r,t). 
\]
By using \eqref{eq:spectral-rep-density}, 
it is not difficult to see that 
\begin{align*}
\rho_1(re^{i\theta}) 
&= \frac{1}{\pi(1-r^2)^2} \times 
\left\{1 -
\frac{\frac{\left(1-r^2\right)^2}{1+2 r^2 \cos 2\theta
+r^4}}{\left(\frac{\pi}{2}+\arctan \left(\frac{2r}{1-r^2}
 \cos \theta\right)\right)^2}\right\} \\
&=\begin{cases}
\frac{1}{\pi(1-r^2)^2} \times 
\disp \left\{1 - \frac{1}{4\pi \cos^2 \theta}(1-r^2)^2\right\} 
+ O((1-r^2)^3) & \theta \in (-\frac{\pi}{2}, \frac{\pi}{2}) \\[2mm]
\disp 1 - (\frac{2}{\pi})^2 & \theta =\pm \frac{\pi}{2} \\[2mm]
\disp \frac{1}{12 \cos^2 \theta} (1-r^2)^2 + O((1-r^2)^3) &
   \theta \in
 (-\pi, \pi) \setminus [-\frac{\pi}{2}, \frac{\pi}{2}], 
\end{cases}
\end{align*}
as $r \to 1$. See Fig.~\ref{fig:fig2} for
$\rho_1(re^{i\theta})$. 
\begin{figure}[htbp]
\begin{center}
\includegraphics[width=0.7\hsize]{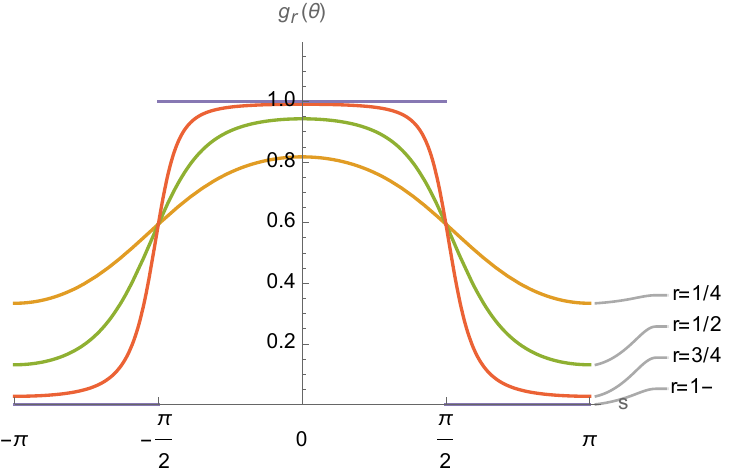}
\caption{The graph of $g_r(\theta) = \pi(1-r^2)^2 \rho_1(r e^{i\theta})$ for $r=1/4, 1/2,
 3/4, 1-$. The common intersection point of all the graphs is
 $1-(2/\pi)^2 \approx 0.594715$ at $\theta=\pm \pi/2$. }
\label{fig:fig2} 
\end{center}
\end{figure}
While the expected number of zeros is $O((1-r^2)^{-2})$ as $r \to 1$, 
it is extremely small; indeed, $O(1)$ in the direction of $e^{i\theta}$ with $\theta \in (-\pi,\pi)
\setminus [-\pi/2,\pi/2]$, where the spectral density
vanishes. This implies that the boundary behavior of the
$1$-intensity of zeros heavily depends on the boundary value
of the spectral density \cite{Shirai25+}. 
See the right panel in Fig.~\ref{fig:fig1}. 

For the simulation shown in the right panel of 
Fig.~\ref{fig:fig1}, we approximate $X_F(z)$ using the diagonalization of the
coefficient Gaussian process. 
We take the unique positive-definite square root $A$ of the positive-definite,
symmetric matrix $K_F = (\gamma(i-j))_{i,j = 0}^{\infty}$, 
i.e.,  $K_F = A^2$. Then, 
$\Xi = \{\xi_i\}_{i=0}^{\infty}$ can be constructed as
follows: 
\[
\xi_i := \sum_{j=0}^{\infty} A_{ij} \zeta_j, 
\]
where $\{\zeta_j\}_{j=0}^{\infty}$ are i.i.d complex
standard normal. 
Here, we approximate $X_F(z) = \sum_{i=0}^{\infty} \xi_i
z^i$ by the double sum 
\[
 X_F^{(N)}(z) = \sum_{i=0}^N \left(\sum_{j=0}^N A_{ij} \zeta_j\right) z^i
= \sum_{j=0}^N \zeta_j \left(\sum_{i=0}^N
A_{ij}z^i\right). 
\]
It is easy to see that the covariance kernel is given by 
\[
 K_F^{(N)}(z,z) = \sum_{i,k=0}^N (A|_{[0,N]}^2)_{ik} z^i
 \zbar^k,  
\]
where $A|_{[0,N]}$ is the submatrix whose rows and columns are
restricted on the indices $\{0,1,\dots, N\}$. Let $P^{(N)}$ be the
projection operator that maps $(x_0,x_1,\dots)$ to
$(x_0,x_1,\dots,x_N, 0,0,\dots)$, i.e., the restriction onto
the set $\{0,1,\dots,N\}$. Then,  
\[
K_F^{(N)}(z,z) = \bra P^{(N)} A P^{(N)} A P^{(N)} \z, \z \ket =
\|P^{(N)}AP^{(N)} \z\|^2 \le \|A \z\|^2 = K_F(z,z), 
\]
where $\z=\{1,z,z^2,\dots\}$ with $z \in \D$, from which we
can see that $\sup_N
E[|X_F^{(N)}(z)|^2] = \sup_N K_F^{(N)}(z,z)$ is locally
integrable. Since $K_F^{(N)}(z,w) \to K_F(z,w)$ for every
$z,w\in \D$, by Propositions~\ref{prop:tightness} and 
\ref{prop:zeros-converge}, we have that $\{X_F^{(N)}(z)\}_{z
\in \D}$ converges in law to $\{X_F(z)\}_{z \in \D}$, and
hence the zero process $\cZ_{X_F^{(N)}}$ converges 
in law to $\cZ_{X_F}$ as $N \to \infty$. 

The zeros of $X_F^{(N)}(z)$ were computed repeatedly for $N =
200$ as shown in Fig.~\ref{fig:fig3}. 
In the left half-plane, the zeros are almost evenly
distributed near the unit circle, whereas in the right
half-plane, the zeros are more randomly distributed, similar
to the previous examples. 
On average, the number of zeros inside and outside the unit
disk is approximately equal both in the right and left
half-plane. 
\begin{figure}[htbp]
\begin{center}
\includegraphics[width=1\hsize]{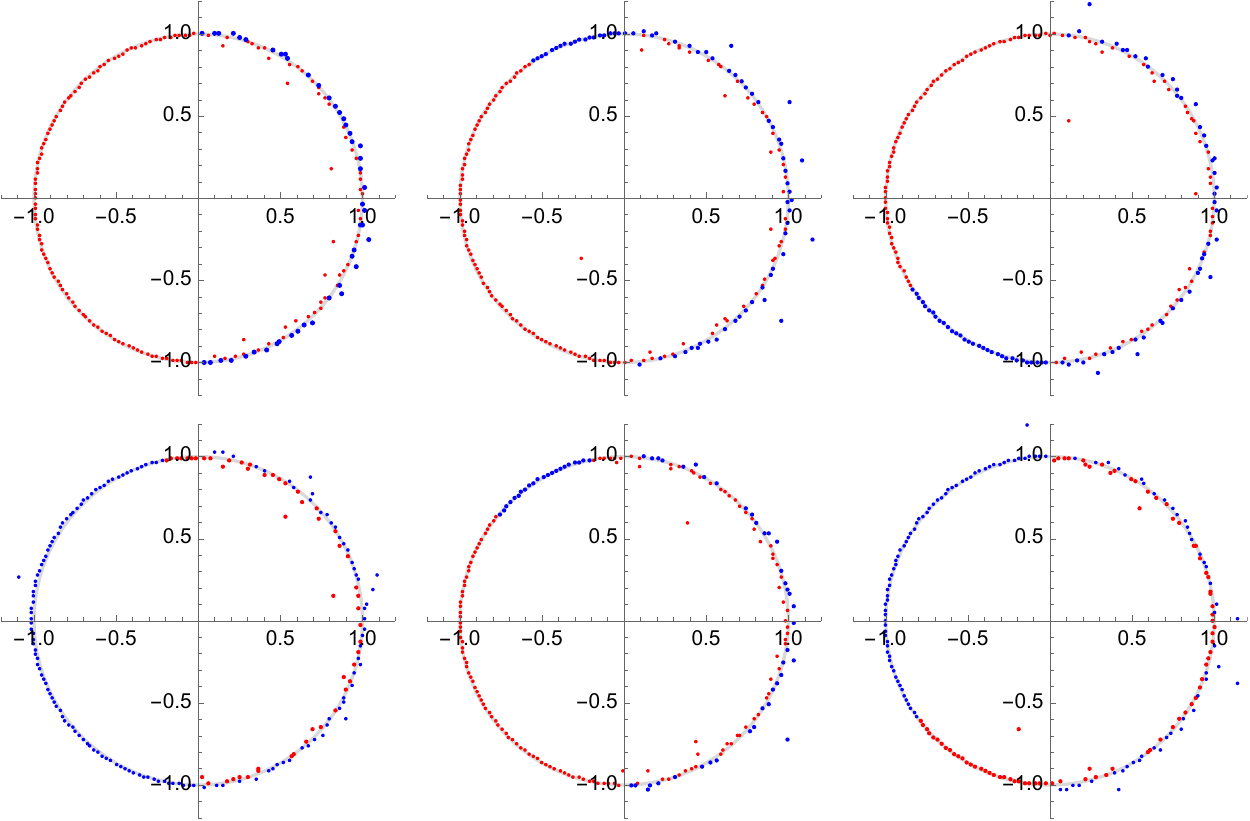}
\end{center}
\caption{The zeros of $X_F^{(N)}(z)$ were computed for $N
 =200$. The red points indicate zeros inside the unit disk,
 while the blue points indicate zeros outside the unit
 disk. The patterns of red and blue points seem to random. }
\label{fig:fig3}
\end{figure}
\begin{figure}[htbp]
\begin{center}
\includegraphics[width=1\hsize]{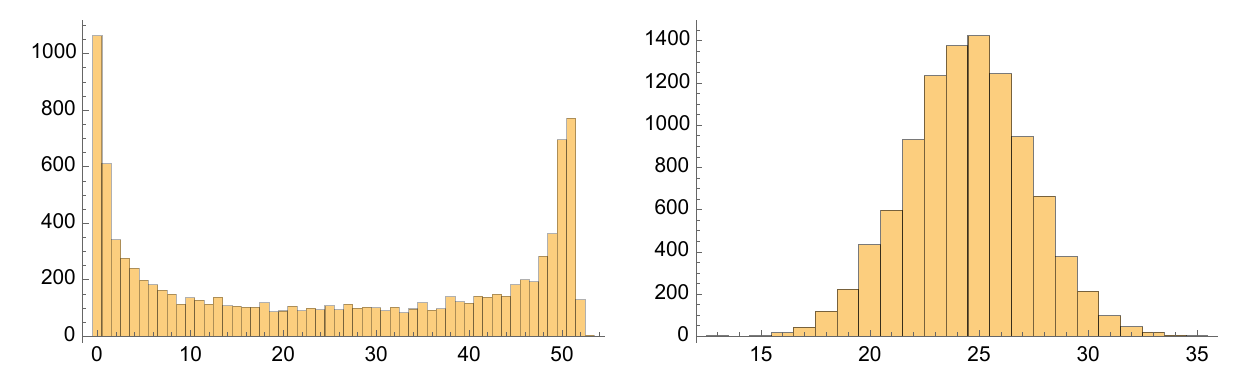}
\end{center}
\caption{The zeros of $X_F^{(100)}(z)$ are computed $10,000$ 
 times. The left (resp. right) panel shows the number of zeros inside the
 unit disk in the left-half (resp. right-half) plane. } 
\label{fig:fig4}
\end{figure}

However, the variance for the left-half
plane appears to be larger than that for the right half-plane. 
To see this situation further, we computed the zeros of $X_F^{(100)}(z)$ $10,000$
 times. The left (resp. right) panel of Fig.~\ref{fig:fig4} shows the number of zeros inside the
 unit disk in the left-half (resp. right-half) plane. 
The central limit theorem seem to hold for the number of zeros
 inside $\D$ in the right-half plane. On the other hand, 
the number of zeros inside $\D$ in the left-half plane
 behaves differently.

\section{Spectral representation of $n$-point correlation function}\label{sec:n-correlation}

Using the Edelman-Kostlan formula, the 1-correlation
function is expressed in terms of the covariance kernel. In
Section~2, we provided an integral representation \eqref{eq:spectral-rep-density} of the
$1$-correlation function involving the spectral measure
$F(dt)$. In this section, we extend this approach to derive an 
integral representation of the $n$-point correlation function for
the zeros of the GAF, also in terms of the spectral measure
$F(dt)$. 

We define the conditional kernel $K^a(z,w)$ by 
\begin{equation}
 K^a(z,w) = K(z,w) - \frac{K(z,a)K(a,w)}{K(a,a)}
= \frac{1}{K(a,a)} 
\begin{vmatrix}
 K(z,w) & K(z,a) \\
 K(a,w) & K(a,a)
\end{vmatrix}
\label{eq:palm_kernel}
\end{equation}
whenever $K(a,a)>0$, and inductively, $K^{a_1,a_2,\dots,a_n}(z,w)
= (K^{a_1,a_2,\dots,a_{n-1}})^{a_n}(z,w)$. Note that
$K^{a_1,a_2,\dots,a_n}(z,w)$ can also be represented as the ratio
of determinants, i.e., 
\begin{equation}
K^{a_1,a_2,\dots, a_n}(z,w) 
= \Big(\det(K(a_j,a_k))_{j,k=1}^n \Big)^{-1}
\det(K(a_j,a_k))_{j,k=0}^n\Big|_{\text{$a_0=z$ or $w$}},  
\label{eq:palm_kernel-2}
\end{equation}
where ``$a_0=z$ or $w$'' means that $K(a_0,a_k) = K(z,a_k)$ and 
$K(a_j,a_0) = K(a_j,w)$ for $j,k=1,2,\dots,n$, and
$K(a_0,a_0) = K(z,w)$. When we denote by $\cH_K$ the reproducing kernel
Hilbert space corresponding to $K$,
$\cH_{K^{a_1,\dots,a_n}}$ is given by 
\[
\cH_{K^{a_1,\dots,a_n}} = \{f \in \cH_K : f(a_j) = 0 \text{
for $j=1,2\dots,n$}\}. 
\]
This corresponds to the conditional GAF $\{X(z)\}_{z \in
\D}$ given that $X(a_1)=X(a_2)= \cdots = X(a_n)=0$. 

Now we recall the following formula for the $n$-point correlation
function of the zeros of GAF (cf. \cite[Corollary 3.4.2]{HKPV09}). 
For our convenience, we use the same formula in a slightly
different form as presented in \cite[Proposition 6.1]{Shirai12}, 
\begin{prop}\label{prop:prop6.1}
The $n$-point correlation function of the zeros of GAF with covariance kernel $K(z,w)$ 
is given by the formula 
\[
 \rho_n(z_1,z_2,\dots,z_n) = \frac{\per(\partial_z \partial_{\wbar} K^{z_1,z_2,\dots,z_n}(z_j,z_k))_{j,k=1}^n}{\det (\pi K(z_j,z_k))_{j,k=1}^n} 
\]
for distinct $z_1,z_2,\dots,z_n \in D$ with 
$\det (K(z_j,z_k))_{j,k=1}^n > 0$, where $\per A$ 
is the permanent of an $n$ by $n$ matrix 
$A = (a_{jk})_{j,k=1}^n$ defined by 
\[
 \per A = \sum_{\sigma \in \cS_n} \prod_{j=1}^n a_{j \sigma(j)}, 
\]
where $\cS_n$ is the symmetric group of order $n$.
\end{prop}

This formula follows from the fact that 
\begin{align*}
\lefteqn{\rho_n(z_1,z_2,\dots,z_n)}\\
&= \E[|X'(z_1)X'(z_2)\cdots X'(z_n)|^2 \
 | \ X(z_1)=X(z_2)=\cdots =X(z_n)=0] 
\end{align*}
and the second moment of the product of \textit{complex} Gaussian random
variables is equal to the permanent of the covariance
matrix, i.e., if $Y=(Y_1,Y_2,\dots,Y_n)$ is a centered
Gaussian vector with covariance matrix $K=(K_{jk})_{j,k=1}^n$, then 
\[
 \E[|Y_1Y_2\cdots Y_n|^2] = \per K. 
\]

For $\a=(a_1,\dots,a_n)$, let $V(\a)$ be the Vandermonde determinant defined by 
\[
V(\a) = V(a_1,a_2,\dots,a_n) 
:= \det(a_j^{k-1})_{j,k=1}^n= \prod_{1\le j < k \le n} (a_k - a_j).  
\]
Later, we will also use it in the form 
\begin{equation}
\prod_{j,k=1 \atop{k\not=j}}^n (a_k-a_j) = (-1)^{{n \choose
 2}} V(\a)^2. 
\label{eq:prodap-aj} 
\end{equation}
We recall the Cauchy determinant identity 
\begin{equation}
 \det\Big(\frac{1}{1-a_jb_k}\Big)_{j,k=1}^n
= V(\a) V(\b) \prod_{j,k=1}^n \frac{1}{1-a_jb_k}. 
\label{eq:cauchy_identity}
\end{equation}
For $\a=(a_1,a_2,\dots,a_n)$, $\t=(t_1,t_2,\dots,t_n)$, and
$\z=(z_1,z_2,\dots,z_m)$ (possibly with different $n$ and $m$), we define 
\[
 D(\a,\t) = \det(s(a_j,t_k))_{j,k=1}^n, \quad 
S(\z,\t) = \prod_{j=1}^m \prod_{k=1}^n s(z_j,t_k). 
\]
In what follows in this section, the elements of $\a$ and $\z$ are assumed
to belong to $\D$. 

From \eqref{eq:kzt} and the Cauchy identity \eqref{eq:cauchy_identity}, we have 
\begin{equation}
 D(\a,\t) = V(\a) V(e^{-it_1},\dots,e^{-it_n}) S(\a,\t).  
\label{eq:cauchy_identity2}
\end{equation}

\begin{lem}\label{lem:detK} 
For $\a=(a_1,a_2,\dots,a_n)$ and $\t=(t_1,t_2,\dots,t_n)$,
 we have 
\begin{equation}
\det(K(a_j,a_k))_{j,k=1}^n 
= |V(\a)|^2 \mu_{\a}(\T^n), 
\label{eq:detS}
\end{equation}
where 
\begin{equation}
\mu_{\a}(d\t) = \frac{1}{n!} 
|V(e^{-it_1}, e^{-it_2},\dots, e^{-it_n})|^2
\cdot |S(\a,\t)|^2 F^{\otimes n}(d\t). 
\label{eq:muadt}
\end{equation}
\end{lem}
\begin{proof} By \eqref{eq:Szw}
and the symmetrization of the variable $\t$, we see that 
\begin{align*}
\det(K(a_j,a_k))_{j,k=1}^n
&= \sum_{\sigma \in \cS_n} \sgn(\sigma) 
\int_{\T^{n}} \prod_{k=1}^n s(a_{\sigma(k)},t_k) \cdot 
\overline{\prod_{k=1}^n s(a_k,t_k)}
F^{\otimes n}(d\t) \\
&= \int_{\T^{n}} \det(s(a_j,t_k))_{j,k=1}^n
\overline{\Big(\prod_{k=1}^n s(a_k,t_k)\Big)}
F^{\otimes n}(d\t) \\
&=\frac{1}{n!} 
\int_{\T^{n}} 
\det(s(a_j,t_k))_{j,k=1}^n
\overline{\det(s(a_j,t_k))_{j,k=1}^n}  
F^{\otimes n}(d\t) \\
&=\frac{1}{n!} 
\int_{\T^{n}} |D(\a,\t)|^2 
F^{\otimes n}(d\t). 
\end{align*}
The assertion follows directly from the identity \eqref{eq:cauchy_identity2}. 
\end{proof}

\begin{rem}\label{rem:DPPonT}
Let $\cU(n)$ denote the compact Lie group of $n \times n$
 unitary matrices and let $P_n$ be the Haar measure on it. 
The pair $(\cU(n), P_n)$ is called the circular unitary ensemble CUE$_n$. 
The eigenvalues of CUE$_n$ are known to form the $n$-point DPP 
on $\T$ associated with the rank $n$ projection kernel 
\[
 \tilde{k}_n(\theta, \phi) 
=\sum_{k=0}^{n-1}  e^{i k\theta} \overline{e^{ik\phi}}
= u(\theta) \frac{\sin
 \frac{n(\theta-\phi)}{2}}{\sin\frac{\theta-\phi}{2}}
 \overline{u(\phi)} 
=: u(\theta) k_n(\theta, \phi)
 \overline{u(\phi)}, 
\]
where $u(\theta) = e^{i (n-1)\theta/2}$ and with the background
 measure being the Lebesgue probability measure on $\T$. 
The probability measure on $\T^n$ with density $\mu_{\a}(\t)/\mu_{\a}(\T^n)$ can be
 regarded as the $n$-point DPP associated with the
 correlation kernel $k_n(\theta,\phi)$ and the background measure 
\[
|S(\a,t)|^2 F(dt) = \Big(\prod_{j=1}^n \frac{1}{1-r_j^2} \Big) 
\Big(\prod_{j=1}^n P(r_j, \phi_j - t) \Big) F(dt), 
\]
where $a_j = r_j e^{i\phi_j}$ and $P(r,\theta)$ is the
 Poisson kernel in \eqref{eq:poisson-kernel}. 
We denote this $n$-point DPP on $\T$ by $\P_{\a}^{(n)}$. 
\end{rem}

\begin{lem}\label{lem:palmn} 
For $\a=(a_1,a_2,\dots,a_n)$ and $\tilde{\t}
 =(t_0,t_1,\dots,t_n)$, 
\begin{align*}
\lefteqn{\det(K(a_j,a_k))_{j,k=1}^n \cdot 
K^{a_1,\dots,a_n}(z,w)} \\
&= |V(\a)|^2
\prod_{i=1}^n (z-a_i)\overline{(w-a_i)} \int_{\T^{n+1}} 
S(z,\tilde{\t}) \overline{S(w,\tilde{\t})}
 \mutilde_{\a}(d\tilde{\t}) \\
&=|V(\a)|^2 \mutilde_{\a}(\T^{n+1})
\prod_{i=1}^n (z-a_i)\overline{(w-a_i)} \cdot K_{\a}^{(n+1)}(z,w) 
\end{align*}
where 
\[
\mutilde_{\a}(d\tilde{\t}) 
= \frac{1}{(n+1)!}|V(e^{-it_0}, e^{-it_1}, \dots, e^{-it_n})|^2 
|S(\a,\tilde{\t})|^2
 F^{\otimes(n+1)}(d\tilde{\t}). 
\]
and 
\[
K_{\a}^{(n+1)}(z,w) = \E_{\a}^{(n+1)}[S(z,\tilde{\t})
 \overline{S(w,\tilde{\t})}],  
\]
where $\E_{\a}^{(n+1)}$ is the expectation with respect to
 $\P_{\a}^{(n+1)}$ defined as in Remark~\ref{rem:DPPonT}. 
\end{lem}
\begin{proof} 
As in the previous calculation in the proof of
 Lemma~\ref{lem:detK}, by \eqref{eq:palm_kernel-2}, it follows that 
\begin{align*}
Q_n(z,w)
&:=\det(K(a_j,a_k))_{j,k=1}^n K^{a_1,\dots,a_n}(z,w) \\
&=\det(K(a_j,a_k))_{j,k=0}^n\Big|_{\text{$a_0=z$ or $w$}} \\
&=\frac{1}{(n+1)!} 
\int_{\T^{n+1}} 
\det(s(a_i,t_j))_{i,j=0}^n\Big|_{a_0=z} \cdot 
\overline{\det(s(a_i,t_j))_{i,j=0}^n}\Big|_{a_0=w} 
F^{\otimes n+1}(dt_0 d\t). 
\end{align*}
Setting $\tilde{\a} = (a_0,\a)$, by
 \eqref{eq:cauchy_identity}, we have 
\begin{align*}
\det(s(a_j,t_k))_{j,k=0}^n \Big|_{a_0=z}
&= V(\tilde{\a}) V(e^{-it_0}, e^{-it_1}, \dots, e^{-it_n}) S(\tilde{\a}, \tilde{\t})\Big|_{a_0=z} \\
&= \prod_{j=1}^n (a_j-z) \cdot V(\a) 
V(e^{-it_0}, e^{-it_1}, \dots, e^{-it_n}) 
S(z,\tilde{\t}) S(\a,\tilde{\t}).  
\end{align*}
Hence, we obtain 
\begin{align*}
Q_n(z,w)
&=\frac{1}{(n+1)!} |V(\a)|^2
\prod_{j=1}^n (z-a_j)\overline{(w-a_j)} \\
&\quad \times \int_{\T^{n+1}} 
S(z,\tilde{\t}) \overline{S(w,\tilde{\t})}
\cdot |V(e^{-it_0}, e^{-it_1}, \dots, e^{-it_n})|^2 
|S(\a,\tilde{\t})|^2 F^{\otimes(n+1)}(d\tilde{\t}).  
\end{align*}
This completes the proof. 
\end{proof}

Now we compute the $n$-correlation function $\rho_n(\a)$
using Proposition~\ref{prop:prop6.1}, Lemmas~\ref{lem:detK} and \ref{lem:palmn}. 

\begin{thm}\label{thm:n-correlation} 
For pairwise distinct $\a=(a_1,a_2,\dots,a_n) \in \D^n$, 
\begin{align*}
 \rho_n(\a) 
&= 
\frac{1}{\pi^n} |V(\a)|^2 \frac{\mutilde_{\a}(\T^{n+1})^n}{\mu_{\a}(\T^n)^{n+1}} 
\per (K_{\a}^{(n+1)}(a_p,a_q))_{p,q=1}^n.  
\end{align*}
\end{thm}
\begin{proof} By Lemma~\ref{lem:palmn}, we note that 
\begin{align*}
\lefteqn{\partial_z\partial_{\wbar} Q_n(z,w))|_{z=a_p, w=a_q} }\\
&= |V(\a)|^2 \mutilde_{\a}(\T^{n+1}) \left(
\prod_{i=1 \atop{i \not = p}}^n (a_p-a_i)
\overline{\prod_{j=1 \atop{j\not=q}}^n (a_q-a_j)} \right) 
K_{\a}^{(n+1)}(a_p,a_q). 
\end{align*}
Then, we see that 
\begin{align*}
\lefteqn{\per((\partial_z\partial_{\wbar} Q_n(z,w))|_{z=a_p,
 w=a_q})}\\
&= \left\{|V(\a)|^2 \mutilde_{\a}(\T^{n+1})\right\}^{n} 
\sum_{\sigma \in \cS_n} \prod_{p=1}^n 
\left(\prod_{i=1 \atop{i \not = p}}^n (a_p-a_i)
\overline{\prod_{j=1 \atop{j\not=\sigma(p)}}^n
 (a_{\sigma(p)}-a_j)} \right) K_{\a}^{(n+1)}(a_p,a_{\sigma(p)}) \\
&= |V(\a)|^{2n+4} \mutilde_{\a}(\T^{n+1})^n 
\per\big( K_{\a}^{(n+1)}(a_p,a_q)_{p,q=1}^n \big). 
\end{align*}
Here, we used \eqref{eq:prodap-aj} in the second equality.  
Therefore, by Lemma~\ref{lem:detK}, we obtain 
\begin{align*}
 \rho_n(\a) 
&= \frac{\per((\partial_z\partial_{\wbar} K^{a_1,\dots,a_n}(z,w))|_{z=a_p,
 w=a_q})}{\det(\pi K)} 
= \frac{\per((\partial_z\partial_{\wbar} Q_n(z,w))|_{z=a_p,
 w=a_q})}{\pi^n (\det K)^{n+1}} \\
&= \frac{|V(\a)|^2 \mutilde_{\a}(\T^{n+1})^n}{\pi^n \mu_{\a}(\T^n)^{n+1}} \per(K_{\a}^{(n+1)}(a_p,a_q)_{p,q=1}^n). 
\end{align*}
This is the desired formula.  
\end{proof}

\begin{rem}
 Theorem~\ref{thm:n-correlation} shows that the
 $n$-correlation function includes the factor of squared Vandermonde determinant
 $|V(\a)|^2$. 
For $\a=(a_1,\dots,a_n)$, we consider a monic polynomial
 $p(z) = \prod_{k=1}^n (z-a_k)$ and suppose it has the coeffcients $e_k, 
 k=0,1,\dots, n-1$, i.e., $p(z) = \sum_{k=0}^{n} e_k z^k$
 with $e_n=1$. Then the transformation $T : \C^n \to
 \C^n$ defined by $T(a_1,a_2,\dots, a_n) = (e_{n-1},e_{n-2},
 \dots, e_0)$ has Jacobian determinant $|V(\a)|^2$
 (cf. \cite{HKPV09}), which is reflected in the theorem. See also
 \cite{NS12, GP17}. 
\end{rem}

\subsection{Reproducing formula} 
We present a multi-dimensional reproducing formula of the
Cauchy type, which appears in the computation in
Section~\ref{subsec:hyperbolicGAF}. 

\begin{lem}
Let 
\[
C(\z,\a) := \prod_{k=1}^n \prod_{j=1}^n \frac{1}{z_k-a_j}. 
\]
Then, for an analytic function $Q(\z)$ that is symmetric in its variables,  
\begin{equation}
\left(\frac{1}{2\pi i}\right)^{n} \oint_{C^n} V(\z)^2 C(\z,\a) Q(\z)
\prod_{k=1}^n dz_k = (-1)^{{n \choose 2}}  n! \cdot Q(\a),  
\label{eq:reproducing-VC} 
\end{equation}
and for pairwise distinct $\a=(a_1,\dots,a_n)$, 
\begin{align}
\lefteqn{
\left(\frac{1}{2\pi i}\right)^{n} \oint_{C^n} V(\z)^2 C(\z,\a) Q(\z)
\prod_{k=1}^n \frac{dz_k}{z_k} } \nonumber \\ 
&= (-1)^{{n \choose 2}}  n! \cdot \prod_{j=1}^n 
 \frac{1}{a_j} \cdot 
\Big\{
Q(\a) 
- \sum_{p=1}^n 
\prod_{k=1 \atop{k\not= p}}^n\frac{a_k}{a_k-a_p}
Q(0,\a^{(p)}) \Big\}, 
\label{eq:reproducing-VC-2} 
\end{align}
where $C = \{z \in \C : |z|=1\}$ and $\a^{(p)} =
 (a_1,\dots,a_{p-1}, a_{p+1},\dots, a_n)$.
\end{lem}
\begin{proof}
Due to the presence of the Vandermonde term, the residue is
 zero unless the points $(z_1, z_2, \dots, z_n)$ are all distinct and
 form a permutation of $\{a_1, \dots, a_n\}$. 
Since the integral is symmetric in $(z_1, z_2, \dots, z_n)$, it suffices
 to consider the residue at $(z_1,z_2,\dots,z_n) =
 (a_1,a_2,\dots,a_n)$ and multiply it by $n!$. The residue
 is given by 
$\disp V(\a)^2 \prod_{k=1\atop{k\not=j}}^n (a_k-a_j)^{-1}
 Q(\a)$ 
and thus, using \eqref{eq:prodap-aj}, we obtain \eqref{eq:reproducing-VC}. 

Similarly, 
the residue is zero unless the points $(z_1, z_2, \dots, z_n)$ are all distinct and
 form a permutation of $\{0, a_1, \dots, a_n\} \setminus
 \{a_p\}$ for some $p=1,2,\dots,n$. Since the
 integral is symmetric in $(z_1, z_2, \dots, z_n)$, it suffices
 to consider the residue at $\z=\a$ or $\z = (0,\a^{(p)})$
 for $p=1,2,\dots,n$ and multiply it by $n!$. 
For $\z=\a$, the first term in \eqref{eq:reproducing-VC-2}
 is obtained in the same manner as above with the extra term
 $\prod_{k=1}^n 1/a_k$. 
Since $\disp V(0,\a^{(p)}) = (\prod_{j=1 \atop{j\not= p}}^n a_j)
 V(\a^{(p)})$, the residue at $\z = (0,\a^{(p)})$ is given
 by 
\begin{align*}
\res_{\z=(0,\a^{(p)})}&=(\prod_{j=1 \atop{j\not= p}}^n a_j)^2 V(\a^{(p)})^2 
\cdot \prod_{j=1}^n \frac{1}{-a_j}
\cdot \prod_{k=1 \atop{k\not=p}}^n 
\prod_{j=1 \atop{j\not=k}}^n \frac{1}{a_k-a_j} 
\cdot \prod_{k=1 \atop{k\not= p}}^n \frac{1}{a_k} \cdot Q(0,\a^{(p)}) \\
&= (-1)^{{n \choose 2}+1} 
\prod_{j=1}^n \frac{1}{a_j}
\prod_{k=1 \atop{k\not=p}}^n \frac{a_k}{a_k-a_p}
\cdot Q(0,\a^{(p)}). 
\end{align*}
Therefore, we obtain \eqref{eq:reproducing-VC-2}. 
\end{proof}

\subsection{The case $F(dt) = dt/(2\pi)$}\label{subsec:hyperbolicGAF}

Let us consider the i.i.d. case, i.e., $F(dt) = dt/(2\pi)$
and $X_F(z) = X_{\hyp}(z)$. 
In the following lemma, by a slight abuse of notation, 
we write $S(\z,\w) = \prod_{j,k=1}^n s(z_j, w_k)
= \prod_{j,k=1}^n (1-z_j \wbar_k)^{-1}$ 
using the same symbol $S$ for the Szeg\H{o} kernel $S(\z,\t)$. 

\begin{lem}\label{lem:volume-formula} 
Suppose $F(dt) = dt/(2\pi)$. Then, 
\[
\mu_{\a}(\T^n) =  \mutilde_{\a}(\T^{n+1}) 
= S(\a,\a), \quad 
K_{\a}^{(n+1)}(z,w) = S(z,w) S(z,\a) \overline{S(w,\a)}. 
\] 
\end{lem}
\begin{proof} 
By \eqref{eq:muadt}, the definition of $\mu_{\a}$, we see
 that 
 \begin{align*}
n! \mu_{\a}(\T^{n}) 
&= \int_{\T^{n}} |V(e^{-it_1}, \cdots, e^{-it_n})|^2 
\prod_{k=1}^n \prod_{j=1}^n \left|\frac{1}{1-a_j
  e^{-it_k}}\right|^2 \prod_{k=1}^n \frac{dt_k}{2\pi} \\
&= \left(\frac{1}{2\pi i}\right)^{n} 
\oint_{C^{n}} |V(z_1^{-1}, \cdots, z_n^{-1})|^2 
\prod_{k=1}^n \prod_{j=1}^n \left|\frac{1}{1-a_j
  z_k^{-1}}\right|^2 \prod_{k=1}^n \frac{dz_k}{z_k} \\
&= (-1)^{{n \choose 2}} \left(\frac{1}{2\pi i}\right)^{n} 
\oint_{C^{n}} V(\z)^2 C(\z,\a)S(\z,\a) \prod_{k=1}^n dz_k, 
 \end{align*}
where we used $\zbar_k = z_k^{-1}$ on $C$ and the following identity: 
\[
 V(z_1^{-1}, \cdots, z_n^{-1}) = (-1)^{{n \choose 2}} 
\frac{V(z_1, \cdots, z_n)}{(z_1 \cdots z_n)^{n-1}}. 
\]
Therefore, by \eqref{eq:reproducing-VC}, we have
 $\mu_{\a}(\T^n) = S(\a,\a)$. 

Similarly as above, setting $\ztilde = (z_0,z_1,\dots,z_n)$
 and $\atilde_z=(z,\a) = (z,a_1,\dots,a_n)$ and using \eqref{eq:reproducing-VC}, we have 
 \begin{align*}
\lefteqn{\mutilde_{\a}(\T^{n+1}) K_{\a}^{(n+1)}(z,w)}\\
&= \int_{\T^{n+1}} S(z,\tilde{\t})
  \overline{S(w,\tilde{\t})} \mutilde_{\a}(d\tilde{\t}) \\
&=\frac{1}{(n+1)!}\int_{\T^{n+1}} |V(e^{-it_0}, \cdots, e^{-it_n})|^2 
\prod_{k=0}^n \prod_{j=1}^n \left|\frac{1}{1-a_j
  e^{-it_k}}\right|^2 \\ 
&\quad \times \prod_{k=0}^n \frac{1}{1-z e^{-it_k}}
  \overline{\frac{1}{1-w e^{-it_k}}} 
\prod_{k=0}^n \frac{dt_k}{2\pi} \\
&= (-1)^{{n+1 \choose 2}} \frac{1}{(n+1)!} 
\left(\frac{1}{2\pi i}\right)^{n+1} 
\oint_{C^{n+1}} V(\ztilde)^2 C(\ztilde, \atilde_z)
  S(\ztilde,\atilde_w)\prod_{k=0}^n dz_k \\
&= S(\atilde_z, \atilde_w). 
\end{align*}
Hence, $K_{\a}^{(n+1)}(z,w) = S(\atilde_z,
 \atilde_w)/\mutilde_{\a}(T^{n+1})$.  
Since the left-hand side is $1$ when $z=w=0$ 
by the definition of $K_{\a}^{(n+1)}(z,w)$, we see that $\mutilde_{\a}(\T^{n+1}) =
S(\tilde{\a}_0,\tilde{\a}_0) = S(\a,\a)$, and thus, 
$K_{\a}^{(n+1)}(z,w) = S(z,w) S(z,\a) \overline{S(w,\a)}$. 
\end{proof}

Now we compute the $n$-correlation function of the zeros for
$X_{\hyp}(z)$. 
By Theorem~\ref{thm:n-correlation} together with Lemma~\ref{lem:volume-formula}, we obtain 
\begin{align*}
\rho_n(\a) 
&= 
\frac{1}{\pi^n} |V(\a)|^2 \frac{1}{S(\a,\a)} 
\per (S(a_p,a_q) S(a_p,\a) \overline{S(a_q,\a)})_{p,q=1}^n \\
&= 
\frac{1}{\pi^n} |V(\a)|^2 S(\a,\a) 
\per (S(a_p,a_q))_{p,q=1}^n \\  
&= 
\frac{1}{\pi^n} 
\det (S(a_p,a_q))_{p,q=1}^n \cdot \per (S(a_p,a_q))_{p,q=1}^n \\
&= 
\frac{1}{\pi^n} 
\det (S(a_p,a_q)^2)_{p,q=1}^n.  
\end{align*}
We used the Cauchy determinant formula
 \eqref{eq:cauchy_identity} for the third equality and 
Borchardt's identity (see \cite[Proposition 5.1.5]{HKPV09})
\[
 \det\Big(\frac{1}{1-a_p b_q}\Big)_{p,q=1}^n \per\Big(\frac{1}{1-a_p
 b_q}\Big)_{p,q=1}^n
= \det\Big(\frac{1}{(1-a_p b_q)^2} \Big)_{p,q=1}^n 
\]
 for the fourth equality. 
We recovered the fact that the zeros of $X_{\hyp}(z)$ is the
 determinantal point process associated with the Bergman
 kernel $S(z,w)^2$. 

\section{Concluding remarks} 

We have discussed GAFs corresponding to three spectral
measures as examples in Section~\ref{sec:examples}. In
particular, for the polynomial approximation of the final example of GAFs, the zeros in
the left half-plane exhibited somewhat different behavior
compared to those in the right half-plane. Further detailed
analysis of this phenomenon is left as a topic for future
research.
The integral representation of the $n$-point correlation
functions in the final section is used only in the proof of
the Peres-Vir\'{a}g theorem. It may also have potential
applications in the detailed analysis of the two-point or
higher order correlation functions of the zeros of GAFs,
particularly for determining whether they exhibit negative
correlations. 
Other topics such as central limit theorems \cite{B15, KN21,
BN22}, hole probability \cite{K06, BNPS18}, rigidity
\cite{GP17, GK21}, and hyperuniformity \cite{HKR22} would
also be interesting to investigate in the context of our setting.

The complex autoregressive model $\mathrm{AR}(1)$ model
$\{X_t\}_{t \in \Z}$ is defined as the following stochastic difference equation
\begin{equation}
 X_t = z X_{t-1} + \zeta_t \quad (t \in \Z) 
\label{eq:discreteOU}
\end{equation}
with i.i.d. complex-valued noise. 
It is known that when $|z| < 1$, there is a unique
stationary solution given by 
\begin{equation}
X_t= X_t(z) = \sum_{k=0}^{\infty} \zeta_{t-k} z^k  \quad (t \in
 \Z). 
\label{eq:stationary-sol}
\end{equation}
If the noise $\{\zeta_t\}_{t \in \Z}$ is an i.i.d. sequence of
$\{-1,1\}$-Bernoulli random variables, for each $t \in
\Z$, the marginal $X_t(\lambda)$ is equal in law to
the so-called Bernoulli convolution $\sum_{k=0}^{\infty} \pm
\lambda^k$, whose distribution $\nu_{\lambda}$ is either absolutely
continuous or purely singular \cite{JW35} and many
properties of $\nu_{\lambda}$ have been studied (cf. \cite{S95, PSS00}).  
If the noise $\{\zeta_t\}_{t \in \Z}$ is an i.i.d. sequence of
standard complex normal random variables, for each $t \in
\Z$, the marginal $\{X_t(z)\}_{z \in \D}$ is equal in law to
the hyperbolic GAF $\{X_{\hyp}(z)\}_{z \in \D}$. 
Then $\{X_t\}_{t \in \Z}$ defines the GAF-valued
stationary process and thus, for each $t$, 
we have the DPP-valued stationary process $\{\cZ_{X_t}\}_{t \in \Z}$.  
If we replace the i.i.d. Gaussian noise with a dependent
Gaussian noise, the stationary solution to
\eqref{eq:discreteOU} still exists when $|z|<1$ and is again given by 
\eqref{eq:stationary-sol}, which 
defines the GAF-valued stationary process. 
The one-dimensional marginal $\{X_t(z)\}_{z \in \D}$ is
identical in law to the one discussed earlier in this paper. 
It would be intriguing to explore the above-mentioned
stationary GAF-valued processes and their zero processes.  

\vskip 1cm 
\noindent
\textbf{Acknowledgment.}  
A part of this study was presented at the conference ``Random
Matrices and Related Topics in Jeju'', held from May 6–10,
2024, on Jeju Island, Korea.
The author would like to sincerely thank Sung-Soo Byun, Nam-Gyu Kang, and Kyeongsik Nam for
organizing such a wonderful conference and their warm hospitality.
The author was supported by JSPS KAKENHI Grant Numbers JP22H05105,
JP23H01077 and JP23K25774, and also supported in part JP21H04432
and JP24KK0060. 

\bibliographystyle{plain}
\bibliography{referenceGAF} 

\end{document}